\documentclass{article}
\usepackage{graphicx} % Required for inserting images
\usepackage{amssymb,amsmath,mathtools,amsthm,amsfonts}
\usepackage{tikz-cd}
\usepackage{mathrsfs}
\newtheorem{lemma}{Lemma}[section]
\newtheorem{theorem}{Theorem}[section]
\newtheorem{corollary}{Corollary}[section]

\newtheorem{example}{Example}[section]

\title{De Bruijn Tori Without Zeros: A Field-Theoretic Perspective}
\author{Ming Hsuan Kang, Yu Hsuan Hsieh}
\date{}

\begin{document}
\def \tr{\mathrm{tr}}

\maketitle

\begin{abstract}
We present an algebraic construction of trace-based De Bruijn tori over finite fields, focusing on the nonzero variant that omits the all-zero pattern. The construction arranges nonzero field elements on a toroidal grid using two multiplicatively independent generators, with values obtained by applying a fixed linear map, typically the field trace.

We characterize sampling patterns as subsets whose associated field elements form an \( \mathbb{F}_p \)-basis, and show that column structures correspond to cyclic shifts of De Bruijn sequences determined by irreducible polynomials over subfields. Recursive update rules based on multiplicative translations enable efficient computation.
\end{abstract}

\section{Introduction}

De Bruijn sequences are classical combinatorial objects that compactly encode all possible strings of a fixed length over a finite alphabet within a single cyclic sequence. Their concise structure and uniform coverage make them foundational in combinatorics and widely useful in applications ranging from coding theory and signal design to robotics, graphics, and pseudo-random number generation \cite{debruijn1946circuits,fredricksen1982survey,alhakim2014survey}.

In this work, we explore a two-dimensional generalization known as the \emph{De Bruijn torus}, where symbols are arranged on a toroidal grid such that prescribed local patterns (e.g., \( m \times n \) subarrays) appear exactly once. This concept arises naturally in areas such as robot localization \cite{caceres2014designing}, visual encoding \cite{galbiati2020debruijn}, and efficient sampling schemes. We present an algebraic construction of such tori using the structure of finite fields, in which each grid position is associated with a distinct multiplicative field element, and values are determined via a fixed linear map—typically the field trace—to project into the base field.

We focus on a \emph{nonzero variant} of the De Bruijn torus, which omits the all-zero pattern by restricting attention to the multiplicative group \( \mathbb{F}_{p^n}^\times \). While the all-zero vector is excluded, this omission does not hinder any practical purpose; in applications such as sampling, coding, and system identification, it is the uniform coverage of nontrivial patterns that plays the critical role \cite{etzion2008optimality}.

A key concept in our framework is that of a \emph{sampling pattern}—a fixed subgrid of the torus whose toroidal translations span all nonzero vectors in \( \mathbb{F}_p^n \) exactly once. We show that this occurs precisely when the pattern corresponds to a basis of the ambient field over \( \mathbb{F}_p \). We further introduce multiple constructions of such patterns, including those based on Kronecker product structures and recursive extension via linearly independent tiles.

Additionally, we study the recursive structure of the torus, showing how new values in the grid can be computed from previous ones using linear feedback relations derived from the field's multiplicative structure. This generalizes the recurrence relation of one-dimensional De Bruijn sequences and facilitates efficient computation and indexing.

Concrete examples over \( \mathbb{F}_{16} \) are provided to illustrate how the theoretical framework plays out in practice, including the relationship between trace values and irreducibility conditions over subfields. These examples also clarify the interplay between algebraic identities and the visible structure of the torus.

Altogether, our results provide a unified algebraic framework for constructing and analyzing De Bruijn tori, bridging ideas from field theory, combinatorics, and pseudorandomness. Despite omitting the all-zero block, this approach remains robust and broadly applicable in both theoretical and applied contexts.

\section{ Preliminaries}

\subsection{Algebraic Construction of De Bruijn Sequences}

A classical \emph{De Bruijn sequence} of order \( n \) over an alphabet of size \( p \) is a cyclic sequence of length \( p^n \) in which every possible \( n \)-tuple appears exactly once as a consecutive substring. In this subsection, we focus on the \emph{nonzero variant}, which omits the all-zero block and has length \( p^n - 1 \).

When the alphabet is identified with the finite field \( \mathbb{F}_p \), one can exploit the multiplicative and linear structure of \( \mathbb{F}_{p^n} \) to construct such sequences.

Let \( \mathbb{F}_{p^n} \) be a degree-\( n \) extension of \( \mathbb{F}_p \), and fix a primitive element \( \alpha \in \mathbb{F}_{p^n} \), so that
\[
\mathbb{F}_{p^n}^\times = \{ \alpha^0, \alpha^1, \dots, \alpha^{p^n - 2} \}.
\]
Let \( \psi: \mathbb{F}_{p^n} \to \mathbb{F}_p \) be any nonzero \( \mathbb{F}_p \)-linear map. We define the associated \emph{nonzero De Bruijn sequence} as
\[
DB^\psi_\alpha := \left( \psi(\alpha^0),\ \psi(\alpha^1),\ \dots,\ \psi(\alpha^{p^n - 2}) \right).
\]
Sliding windows of length \( n \) through \( DB^\psi_\alpha \) yield all nonzero \( n \)-tuples in \( \mathbb{F}_p^n \) exactly once.

This construction is justified by the coordinate map
\[
\Psi: \mathbb{F}_{p^n} \to \mathbb{F}_p^n, \quad
\Psi(x) = \left( \psi(x),\ \psi(\alpha x),\ \dots,\ \psi(\alpha^{n-1} x) \right),
\]
which is an \( \mathbb{F}_p \)-linear isomorphism when \( \psi \) is nonzero. In particular, the restriction of \( \Psi \) to \( \mathbb{F}_{p^n}^\times \) maps onto the set of nonzero vectors in \( \mathbb{F}_p^n \). Therefore, as \( x \) ranges over \( \mathbb{F}_{p^n}^\times \), the vectors \( \Psi(x) \) enumerate all nonzero \( n \)-tuples in \( \mathbb{F}_p^n \) exactly once.

Moreover, since \( \alpha \) satisfies its minimal polynomial \( f(x) = x^n + c_{n-1} x^{n-1} + \dots + c_0 \), the sequence \( s_i := \psi(\alpha^i) \) satisfies the linear recurrence relation
\[
s_{i+n} = -c_0 s_i - c_1 s_{i+1} - \dots - c_{n-1} s_{i+n-1}.
\]
This allows the entire sequence to be generated efficiently from any initial segment of length \( n \), together with the \emph{feedback vector} \( (-c_0, -c_1, \dots, -c_{n-1}) \), which encodes the recurrence coefficients. This recurrence structure will reappear in our two-dimensional generalization.

Notably, while the construction depends on the choice of the linear map \( \psi \), the resulting sequence \( DB^\psi_\alpha \) is uniquely determined up to a cyclic shift. This is because the recurrence is governed solely by the minimal polynomial of \( \alpha \), and is thus independent of \( \psi \). In fact, for any nonzero linear map \( \psi \), there exists some integer \( r \) such that
\[
DB^\psi_\alpha = DB_\alpha^{[r]},
\]
where \( DB_\alpha := DB^{\tr}_\alpha \) is the standard De Bruijn sequence constructed using the trace map \( \tr = \tr_{\mathbb{F}_{p^n}/\mathbb{F}_p} \), and \( DB_\alpha^{[r]} \) denotes its \( r \)-shifted version.

There are exactly \( p^n - 1 \) such nonzero maps \( \psi \), each yielding a distinct cyclic shift of \( DB_\alpha \). This perspective becomes especially useful when studying De Bruijn tori, where each column corresponds to a translated sampling pattern and thus to a different shift \( DB_\alpha^{[r]} \).

\subsection{From Sequences to Strips}

When the base field is a proper extension \( \mathbb{F}_{p^m} \cong \mathbb{F}_p^m \), each symbol in a one-dimensional De Bruijn sequence over \( \mathbb{F}_p \) can be expanded into a column vector of length \( m \) over \( \mathbb{F}_p \). Arranging these columns horizontally produces a two-dimensional array—known as a \emph{De Bruijn strip}—in which every possible \( m \times n \) binary submatrix (excluding the all-zero block, in the nonzero case) appears exactly once as a consecutive submatrix.

This reinterpretation highlights the combinatorial completeness of De Bruijn sequences in higher dimensions and anticipates the structure of \emph{De Bruijn tori}, which generalize this idea by using multiplicative grid placements and sampling via linear projections.

\begin{example}
Let \(\beta\) be a root of \(x^2 + x + 1\) in \(\mathbb{F}_2[x]\), so that \(\mathbb{F}_4 = \mathbb{F}_2(\beta)\). Let \(\alpha\) be a root of \(x^2 + x + \beta\) in \(\mathbb{F}_{4^2}\). Then \(\alpha\) generates \( \mathbb{F}_{16}^\times \). The feedback vector corresponding to the minimal polynomial of \( \alpha \) over \( \mathbb{F}_4 \) is \( (\beta, 1) \), and the corresponding De Bruijn sequence is:
\[
\begin{array}{|c|c|c|c|c|c|c|c|c|c|c|c|c|c|c|c|}
\hline
0 & 0 & 1 & 1 & \beta+1 & 1 & 0 & \beta & \beta & 1 & \beta & 0 & \beta+1 & \beta+1 & \beta & \beta+1 \\ 
\hline
\end{array}
\]

Rewriting each symbol as a binary column vector (over \( \mathbb{F}_2 \)) produces the following \emph{De Bruijn strip}:
\[
\begin{array}{|c|c|c|c|c|c|c|c|c|c|c|c|c|c|c|c|}
\hline
0 & 0 & 1 & 1 & 1 & 1 & 0 & 0 & 0 & 1 & 0 & 0 & 1 & 1 & 0 & 1 \\
\hline
0 & 0 & 0 & 0 & 1 & 0 & 0 & 1 & 1 & 0 & 1 & 0 & 1 & 1 & 1 & 1 \\
\hline
\end{array}
\]
Here, every \( 2 \times 2 \) binary submatrix appears exactly once, satisfying the De Bruijn property in two dimensions.
\end{example}

\section{The Algebraic De Bruijn Torus}

\subsection{Construction from Field Geometry}\label{construction}
In this section, we introduce a construction of a De Bruijn torus based on a finite field.
For a prime number $p$, the multiplicative group $\mathbb{F}_{p^n}^{\times}$ of the finite field $\mathbb{F}_{p^n}$, where $p^n-1$ can be expressed as a product of two relatively prime positive integers $s$ and $t$, leads to an $s\times t$ De Bruijn torus.

Let \( p \) be a prime number, and let \( \mathbb{F}_{p^n} \) be a finite field such that \( p^n - 1 = s \times t \), where \( s < t \) and \( \gcd(s, t) = 1 \). Let \( \alpha \in \mathbb{F}_{p^n} \) be a root of a primitive polynomial of degree \( n \), so that \( \alpha \) generates the multiplicative group \( \mathbb{F}_{p^n}^\times \).

We construct an \( s \times t \) toroidal array \( A = (A_{i,j}) \) with entries in \( \mathbb{F}_{p^n}^\times \), defined by
\[
A_{i,j} := \beta^i \gamma^j, \quad \text{where } \beta = \alpha^t, \ \gamma = \alpha^s.
\]
Since \( \gcd(s, t) = 1 \), the map \( (i,j) \mapsto \beta^i \gamma^j \) defines a bijection between the grid positions and the nonzero elements of \( \mathbb{F}_{p^n} \).

Now let \( \psi: \mathbb{F}_{p^n} \to \mathbb{F}_p \) be a nonzero \( \mathbb{F}_p \)-linear map. Applying \( \psi \) entrywise to the array \( A \), we obtain a new array \( B = (B_{i,j}) \) with entries in \( \mathbb{F}_p \), defined by
\[
B_{i,j} := \psi(\beta^i \gamma^j).
\]

We refer to \( B \) as a \textbf{De Bruijn torus} over \( \mathbb{F}_p \).

\subsection{Sampling Patterns and Their Characterization  }

To study the combinatorial behavior of the De Bruijn torus \( B = (B_{i,j}) \), we focus on identifying special subsets of positions on the torus—called \emph{patterns}—that extract useful algebraic information. In particular, we aim to find patterns that uniquely encode all nonzero elements of \( \mathbb{F}_{p^n} \).

\paragraph{Translation and Value Patterns.}
Define the toroidal translation operator
\[
T_{a,b} : \mathbb{Z}_s \times \mathbb{Z}_t \to \mathbb{Z}_s \times \mathbb{Z}_t, \quad
T_{a,b}(i,j) := \left( (i+a) \bmod s,\ (j+b) \bmod t \right).
\]
Given a pattern \( S \subset \mathbb{Z}_s \times \mathbb{Z}_t \), its \emph{value pattern} under shift \( (a,b) \) is defined as
\[
B|_{T_{a,b}(S)} := \left\{ B_{(i+a) \bmod s,\ (j+b) \bmod t} \mid (i,j) \in S \right\} \subset \mathbb{F}_p^n.
\]

\paragraph{Terminology.}
Let \( A|_S := \{ \beta^i \gamma^j \mid (i,j) \in S \} \subset \mathbb{F}_{p^n} \) be the set of field elements associated to the pattern \( S \) on the array \( A \) (prior to applying \( \psi \)).

We define:
\begin{itemize}
  \item \( S \) is a \emph{linearly independent pattern} if the elements of \( A|_S \) are linearly independent over \( \mathbb{F}_p \).
  \item \( S \) is a \emph{basis pattern} if \( A|_S \) forms an \(\mathbb{F}_p\)-basis of \( \mathbb{F}_{p^n} \). This implies \( |S| = n \).
  \item \( S \) is a \emph{sampling pattern} if the value patterns \( B|_{T_{a,b}(S)} \) over all torus translations are all distinct and cover every nonzero vector in \( \mathbb{F}_p^n \) exactly once.
\end{itemize}

That is, \( S \subset \mathbb{Z}_s \times \mathbb{Z}_t \) is a sampling pattern if
\[
\left\{ B|_{T_{a,b}(S)} \mid (a,b) \in \mathbb{Z}_s \times \mathbb{Z}_t \right\}
\]
contains exactly \( p^n - 1 \) distinct nonzero vectors in \( \mathbb{F}_p^n \).

Such a pattern certifies that the torus encodes a complete set of distinct nonzero length-\( n \) vectors over \( \mathbb{F}_p \), thus mimicking the behavior of a classical De Bruijn sequence in a two-dimensional setting.

\begin{theorem}
Let \( B = (B_{i,j}) \) be the De Bruijn torus defined by \( B_{i,j} = \psi(\beta^i \gamma^j) \), where \( \psi: \mathbb{F}_{p^n} \to \mathbb{F}_p \) is a fixed nonzero \(\mathbb{F}_p\)-linear map. Then a pattern \( S\) of size $n$ is a sampling pattern if and only if  it is a basis pattern.
\end{theorem}
\begin{proof}
Define an \(\mathbb{F}_p\)-linear map
\[
\Phi: \mathbb{F}_{p^n} \to \mathbb{F}_p^{S}, \quad
\Phi(z) := \left( \psi(x z) \right)_{x \in A|_S}.
\]
We will show that \(\Phi\) is an isomorphism if and only if \( A|_S \) forms a basis of \( \mathbb{F}_{p^n} \) over \(\mathbb{F}_p\).

\textbf{(\(\Rightarrow\))} Suppose \(\Phi\) is a linear isomorphism. Then the coordinate functionals
\[
z \mapsto \psi(x z), \quad \text{for each } x \in A|_S,
\]
must form an \(\mathbb{F}_p\)-linearly independent set in the dual space \( \mathrm{Hom}_{\mathbb{F}_p}(\mathbb{F}_{p^n}, \mathbb{F}_p) \). Suppose for contradiction that the set \( A|_S \) is linearly dependent, so there exists a nontrivial relation
\[
\sum_{x \in A|_S} a_x x = 0, \quad \text{with } a_x \in \mathbb{F}_p, \text{ not all zero}.
\]
Then for every \( z \in \mathbb{F}_{p^n} \), we have
\[
\sum_{x} a_x \psi(x z) = \psi\left( \left( \sum_{x} a_x x \right) z \right) = \psi(0) = 0,
\]
which implies that the linear combination \( \sum_{x} a_x \psi(x z) \) vanishes identically as a function of \( z \). This contradicts the assumption that the coordinate functionals are linearly independent. Hence, \( A|_S \) must be linearly independent. Since \( |S| = n = \dim_{\mathbb{F}_p} \mathbb{F}_{p^n} \), it follows that \( A|_S \) is a basis.

\textbf{(\(\Leftarrow\))} Conversely, suppose \( A|_S \) is an \(\mathbb{F}_p\)-basis of \( \mathbb{F}_{p^n} \). We prove that \(\Phi\) is an isomorphism by showing it is injective.

Assume \( \Phi(z) = 0 \) for some \( z \in \mathbb{F}_{p^n} \). Then \( \psi(x z) = 0 \) for all \( x \in A|_S \). Since \( A|_S \) is a basis and \( z \ne 0 \), the set \( \{ x z \mid x \in A|_S \} \) also forms a basis of \( \mathbb{F}_{p^n} \) (as multiplication by a nonzero scalar is an automorphism of the vector space). Therefore, \( \psi \) vanishes on a basis and hence on all of \( \mathbb{F}_{p^n} \), which contradicts the assumption that \( \psi \) is nonzero. Thus \( z = 0 \), and \(\Phi\) is injective.

Since both the domain and codomain of \(\Phi\) are \( \mathbb{F}_p \)-vector spaces of dimension \( n \), injectivity implies that \(\Phi\) is an isomorphism.
\end{proof}

\subsection{Construction of Sampling Patterns}

We begin with an explicit construction based on the Kronecker (tensor product) structure of the field extension.

Suppose \( s = p^m - 1 \) for some \( m \mid n \), and define
\[
t = \frac{p^n - 1}{s} = \frac{p^n - 1}{p^m - 1}.
\]
Assume further that \( \gcd(s, t) = 1 \), so that the elements
\[
\beta = \alpha^t \in \mathbb{F}_{p^m}^\times, \quad \gamma = \alpha^s \in \mathbb{F}_{p^n}^\times
\]
generate subgroups whose product spans all of \( \mathbb{F}_{p^n}^\times \), i.e., \( \mathbb{F}_{p^n}^\times = \langle \beta \rangle \times \langle \gamma \rangle \).

Then:
\begin{itemize}
    \item The field \( \mathbb{F}_{p^m} = \mathbb{F}_p(\beta) \) is a degree-\( m \) extension of \( \mathbb{F}_p \).
    \item The field \( \mathbb{F}_{p^n} = \mathbb{F}_{p^m}(\gamma) \) is a degree-\( \frac{n}{m} \) extension of \( \mathbb{F}_{p^m} \), so \( \mathbb{F}_{p^n} = \mathbb{F}_p(\beta, \gamma) \).
\end{itemize}

It follows that the set
\[
\mathcal{B} = \left\{ \beta^i \gamma^j \ \middle|\ 0 \le i < m,\ 0 \le j < \frac{n}{m} \right\}
\]
forms an \(\mathbb{F}_p\)-basis of \( \mathbb{F}_{p^n} \). This is called a \emph{Kronecker basis}.

The corresponding sampling pattern \( S \subset \mathbb{Z}_s \times \mathbb{Z}_t \) is defined by
\[
S = \left\{ (i,j) \in \mathbb{Z}_s \times \mathbb{Z}_t \ \middle|\ 0 \le i < m,\ 0 \le j < \frac{n}{m} \right\}.
\]
Then \( A|_S = \{ \beta^i \gamma^j \mid (i,j) \in S \} = \mathcal{B} \) is an \(\mathbb{F}_p\)-basis, so \( S \) is a \emph{basis pattern}, and thus a valid sampling pattern. Geometrically, it forms a rectangular grid of size \( m \times \frac{n}{m} \), providing a structured example of a sampling pattern.

A more flexible strategy is to build up a sampling pattern from repeated translations of a smaller linearly independent pattern. This method relies on the following lemma.

\begin{lemma}
Let \( V, W \subset \mathbb{F}_{p^n} \) be \(\mathbb{F}_p\)-subspaces such that
\[
\dim_{\mathbb{F}_p}(V) + \dim_{\mathbb{F}_p}(W) \leq n.
\]
Then there exists \( z \in \mathbb{F}_{p^n}^\times \) such that
\[
V \cap zW = \{0\}.
\]
\end{lemma}

\begin{proof}
Let \( \dim V = r \) and \( \dim W = s \), so \( r + s \leq n \). Suppose, for contradiction, that for every \( z \in \mathbb{F}_{p^n}^\times \), we have \( V \cap zW \ne \{0\} \). Then for each such \( z \), there exist nonzero \( v \in V \), \( w \in W \) such that \( z = v/w \).

The number of such quotients is at most
\[
(p^r - 1)(p^s - 1) < p^n - 1,
\]
which is the total number of elements in \( \mathbb{F}_{p^n}^\times \). Hence, not every \( z \in \mathbb{F}_{p^n}^\times \) arises this way, and there exists a \( z \in \mathbb{F}_{p^n}^\times \) such that \( V \cap zW = \{0\} \).
\end{proof}

We can now state the following theorem, which allows us to extend any linearly independent pattern by a suitable translation of another.

\begin{theorem}[Extension of Independent Patterns]
Let \( S_1, S_2 \subset \mathbb{Z}_s \times \mathbb{Z}_t \) be two disjoint linearly independent patterns, of sizes \( m_1 \) and \( m_2 \), respectively. Suppose that the associated field elements \( A|_{S_1} \) and \( A|_{S_2} \) span subspaces \( V_1, V_2 \subset \mathbb{F}_{p^n} \) with
\[
\dim_{\mathbb{F}_p}(V_1) + \dim_{\mathbb{F}_p}(V_2) \le n.
\]
Then there exists a toroidal shift \( (a, b) \in \mathbb{Z}_s \times \mathbb{Z}_t \) such that the translated pattern \( T_{a,b}(S_2) \) is disjoint from \( S_1 \), and the combined pattern
\[
S := S_1 \cup T_{a,b}(S_2)
\]
is linearly independent. In particular, this process may be repeated until a basis pattern of size \( n \) is obtained.
\end{theorem}

\begin{corollary}[Recursive Construction via Repetition]
Let \( S_0 \subset \mathbb{Z}_s \times \mathbb{Z}_t \) be a linearly independent pattern of size \( m \mid n \). Then there exist disjoint toroidal shifts \( (a_1, b_1), \dots, (a_r, b_r) \in \mathbb{Z}_s \times \mathbb{Z}_t \) with \( r = \frac{n}{m} \), such that the union
\[
S := \bigcup_{k=1}^{r} T_{a_k, b_k}(S_0)
\]
forms a basis pattern, and hence a sampling pattern.
\end{corollary}

\subsection{Structure and Trace Patterns in the De Bruijn Torus}

Up to toroidal translation, the combinatorial structure of \( B \) is independent of the choice of \( \psi \), so we may take \( \psi \) to be the field trace:
\[
\psi(x) = \mathrm{tr}_{\mathbb{F}_{p^n}/\mathbb{F}_p}(x) = x + x^p + \cdots + x^{p^{n-1}}.
\]
We call \( B \) the \textbf{trace De Bruijn torus} over \( \mathbb{F}_p \).

Suppose \( m \mid n \) so that \( \mathbb{F}_{p^m} \subset \mathbb{F}_{p^n} \), and let
\[
s = p^m - 1, \quad t = \frac{p^n - 1}{s}, \quad \beta = \alpha^t \in \mathbb{F}_{p^m}^\times, \quad \gamma = \alpha^s.
\]
Furthermore, we assume that \(\gcd(s,t)=1\). 

Define the standard (1-dimensional) trace De Bruijn sequence over \( \mathbb{F}_{p^m}^\times \) as
\[
\mathsf{DB}_\beta := \left( \mathrm{tr}_{\mathbb{F}_{p^m}/\mathbb{F}_p}(\beta^i) \right)_{0 \leq i < s}.
\]
For any integer \( r \), let \( \mathsf{DB}_\beta^{[r]} \) denote the cyclic shift of \( \mathsf{DB}_\beta \) by \( r \) positions, that is,
\[
\mathsf{DB}_\beta^{[r]}(i) := \mathrm{tr}_{\mathbb{F}_{p^m}/\mathbb{F}_p}(\beta^{r+i}), \quad \text{for } 0 \le i < s,
\]
where indices are taken modulo \( s \).

Then the entries of the trace De Bruijn torus \( B \) are given by:
\[
B_{i,j} = \mathrm{tr}_{\mathbb{F}_{p^n}/\mathbb{F}_p}(\beta^i \gamma^j) = \mathrm{tr}_{\mathbb{F}_{p^m}/\mathbb{F}_p}\left( \mathrm{tr}_{\mathbb{F}_{p^n}/\mathbb{F}_{p^m}}(\gamma^j) \cdot \beta^i \right).
\]

\begin{theorem}[Structure of the Trace De Bruijn Torus Columns]
Let \( B_{i,j} \) be as above. Then each column is determined by the trace \( \mathrm{tr}_{\mathbb{F}_{p^n}/\mathbb{F}_{p^m}}(\gamma^j) \in \mathbb{F}_{p^m} \), and falls into one of the following:
\begin{enumerate}
    \item If \( \mathrm{tr}_{\mathbb{F}_{p^n}/\mathbb{F}_{p^m}}(\gamma^j) = 0 \), the column is identically zero.
    \item If \( \mathrm{tr}_{\mathbb{F}_{p^n}/\mathbb{F}_{p^m}}(\gamma^j) = \beta^r \in \mathbb{F}_{p^m}^\times \), then the column is the cyclic shift \( \mathsf{DB}_\beta^{[r]} \).
\end{enumerate}
\end{theorem}

\paragraph{Frobenius Orbits and Column Multiplicity.}
Two columns \( j, j' \) are equal if
\[
\mathrm{tr}_{\mathbb{F}_{p^n}/\mathbb{F}_{p^m}}(\gamma^j) = \mathrm{tr}_{\mathbb{F}_{p^n}/\mathbb{F}_{p^m}}(\gamma^{j'}).
\]
This always holds when \( \gamma^j \) and \( \gamma^{j'} \) are Frobenius conjugates over \( \mathbb{F}_{p^m} \), i.e., \( \gamma^{j'} = (\gamma^j)^{p^{mi}} \). Therefore, the Galois group \( \mathrm{Gal}(\mathbb{F}_{p^n}/\mathbb{F}_{p^m}) \) acts on the columns, and each orbit contributes to (at most) one distinct column.

\paragraph{Polynomial Interpretation.}
This Galois structure corresponds to the factorization:
\[
x^t - 1 = \prod_{i=0}^{t-1} (x - \gamma^i) = \prod_{\mathcal{O}} m_{\mathcal{O}}(x),
\]
where each \( m_{\mathcal{O}}(x) \in \mathbb{F}_{p^m}[x] \) is the minimal polynomial of a Frobenius orbit \( \mathcal{O} \subset \langle \gamma \rangle \). Each such polynomial has the form:
\[
m_{\mathcal{O}}(x) = x^d - \beta^r x^{d-1} + \cdots,
\]
with \( -\beta^r \) being the trace of any \( \gamma^j \in \mathcal{O} \).

\begin{theorem}[Occurrence Criterion for \( \mathsf{DB}_\beta^{[r]} \)]
Let \( \beta \in \mathbb{F}_{p^m}^\times \), and let \( \mathsf{DB}_\beta \) be the standard trace De Bruijn sequence of length \( s = p^m - 1 \). Then the shifted sequence \( \mathsf{DB}_\beta^{[r]} \) appears as a column in the trace De Bruijn torus if and only if \( -\beta^r \) occurs as the second-leading coefficient of some irreducible factor of \( x^t - 1 \in \mathbb{F}_{p^m}[x] \), where \( t = \frac{p^n - 1}{p^m - 1} \).

Moreover, the total number of columns equal to \( \mathsf{DB}_\beta^{[r]} \) is equal to the sum of degrees of all such irreducible factors:
\[
\#\left\{ j : \mathrm{tr}_{\mathbb{F}_{p^n}/\mathbb{F}_{p^m}}(\gamma^j) = \beta^r \right\} 
= \sum_{\substack{m_{\mathcal{O}}(x) \mid x^t - 1 \\ \text{second-leading coeff } = -\beta^r}} \deg m_{\mathcal{O}}(x).
\]
\end{theorem}

\subsection{Special Case: Degree-2 Extension and Irreducible Polynomials}

We now consider a special case of the trace De Bruijn torus when \( n = 2m \), so that \( \mathbb{F}_{2^n} \) is a quadratic extension of \( \mathbb{F}_{2^m} \). Let \( \gamma \in \mathbb{F}_{2^n}^\times \) be an element of order \( 2^n + 1 \), and let \( \beta \in \mathbb{F}_{2^m}^\times \) be a fixed generator.

The trace map from \( \mathbb{F}_{2^n} \) to \( \mathbb{F}_{2^m} \) is given by
\[
\mathrm{tr}_{\mathbb{F}_{2^n}/\mathbb{F}_{2^m}}(x) = x + x^{2^m}.
\]
For powers of \( \gamma \), we have
\[
\mathrm{tr}(\gamma^j) = \gamma^j + \gamma^{-j},
\]
since \( \gamma^{2^m} = \gamma^{-1} \) when \( \gamma \) has order \( 2^m + 1 \).

The trace vanishes only when \( \gamma^j = \gamma^{-j} \), which implies \( \gamma^{2j} = 1 \). Given that \( \gamma \) has odd order \( 2^n + 1 \), the only solution is \( j = 0 \). Thus, zero trace occurs uniquely.

Now consider \( j \ne 0 \). Then the elements \( \gamma^j \) and \( \gamma^{-j} \) are distinct and satisfy
\[
x^2 + \mathrm{tr}(\gamma^j)x + 1 = 0.
\]
The trace value \( \mathrm{tr}(\gamma^j) \in \mathbb{F}_{2^m} \) is the negative of the second-leading coefficient of the minimal polynomial of \( \gamma^j \) over \( \mathbb{F}_{2^m} \), and hence corresponds to the trace values determining columns in the torus.

Among the \(2^{m} + 1 \) powers of \( \gamma \), the nonzero indices \( j \ne 0 \) fall into \( 2^{m-1} \) distinct pairs \( \{j, -j\} \), each of which gives the same trace. Hence, there are exactly \( 2^{m-1} \) distinct nonzero trace values of the form \( \gamma^j + \gamma^{-j} \), along with the one zero trace value from \( j = 0 \).

To determine which values in \( \mathbb{F}_{2^m} \) occur as such traces, consider for each \( a \in \mathbb{F}_{2^m} \) the polynomial
\[
P_a(x) := x^2 + a x + 1.
\]
This polynomial splits over \( \mathbb{F}_{2^m} \) if and only if \( a = y + y^{-1} \) for some \( y \in \mathbb{F}_{2^m}^\times \). The map \( y \mapsto y + y^{-1} \) is 2-to-1 (except when \( y = y^{-1} \), i.e., \( y = 1 \)), and hence the number of distinct values \( a \in \mathbb{F}_{2^m} \) for which \( P_a(x) \) is irreducible is exactly \( 2^{m-1} \).

Therefore, the set of \( \beta^i \in \mathbb{F}_{2^m}^\times \) for which the column shift \( \mathsf{DB}_\beta^{[i]} \) appears in the torus corresponds precisely to those \( \beta^i \) for which
\[
x^2 + \beta^i x + 1 \text{ is irreducible over } \mathbb{F}_{2^m}.
\]
These shifts arise from the trace values \( \gamma^j + \gamma^{-j} \), and the corresponding De Bruijn columns appear twice, once for each of \( j \) and \( -j \).

\begin{theorem}\label{thm:occurrence-criterion}
Let \( \beta \in \mathbb{F}_{2^m}^\times \) and \( \mathsf{DB}_\beta \) be the standard trace De Bruijn sequence. Then the shifted sequence \( \mathsf{DB}_\beta^{[i]} \) appears as a column in the trace De Bruijn torus if and only if the polynomial
\[
x^2 + \beta^i x + 1
\]
is irreducible over \( \mathbb{F}_{2^m} \).

Moreover, such a column occurs with multiplicity two, corresponding to the two indices \( j \) and \( -j \in \mathbb{Z}/t\mathbb{Z} \) satisfying
\[
\mathrm{tr}_{\mathbb{F}_{2^n}/\mathbb{F}_{2^m}}(\gamma^j) = \beta^i.
\]
\end{theorem}

\section{Recursive Structure and Efficient Updates}
\subsection{ Local Translation Rules }
The recursive structure of De Bruijn tori allows value patterns to be computed efficiently through local multiplicative translations. Given a sampling pattern \( S \subset \mathbb{Z}_s \times \mathbb{Z}_t \) corresponding to field elements \( A|_S = \{a_1, \dots, a_n\} \subset \mathbb{F}_{p^n}^\times \), and a nonzero linear map \( \psi: \mathbb{F}_{p^n} \to \mathbb{F}_p \), the value pattern observed at any torus position \( x \in \mathbb{F}_{p^n}^\times \) is
\[
(\psi(x a_1), \dots, \psi(x a_n)) \in \mathbb{F}_p^n.
\]
A shift to position \( x y \) for some \( y \in \mathbb{F}_{p^n}^\times \) corresponds to translating the pattern by \( y \), which affects each element \( a_j \) as \( y a_j \). Expressing each \( y a_j \) as a linear combination of the original basis elements \( a_i \), say
\[
y a_j = \sum_{i=1}^n c_{i,j} a_i,
\]
we can compute the new values by linearity:
\[
\psi(x y a_j) = \sum_{i=1}^n c_{i,j} \psi(x a_i).
\]
This results in a fixed linear transformation (determined by \( y \)) applied to the previous value pattern. The matrix \( C_y = (c_{i,j}) \) depends only on the relative shift \( y \), not the absolute position \( x \), making this rule globally consistent.

\subsection{Comparison to One-Dimensional Recurrences  }

This recursive update rule generalizes the recurrence relation in classical De Bruijn sequences. In the 1D case, the sampling pattern is \( S = \{1, \alpha, \alpha^2, \dots, \alpha^{n-1}\} \), and a shift by \( \alpha \) leads to the pattern \( \{\alpha, \alpha^2, \dots, \alpha^n\} \). The last term \( \alpha^n \) is expressed as a linear combination of the previous ones using the minimal polynomial of \( \alpha \), yielding a standard recurrence.

While in the 2D torus case the coefficients \( c_{i,j} \) do not directly arise from a minimal polynomial, the recurrence structure persists. The field multiplication induces a consistent linear update rule on value patterns, making the torus a true two-dimensional analogue of De Bruijn sequences.

\subsection{Practical Considerations: Efficient Local Updates}\label{eff-update}

In practice, the most common shifts are one-step translations in the four cardinal directions (left, right, up, down), corresponding to multiplication by \( \beta^{\pm 1} \) or \( \gamma^{\pm 1} \). For these translations, a well-structured sampling pattern enables fast local updates.

A particularly efficient shape is a rectangular \( m \times \frac{n}{m} \) pattern corresponding to a Kronecker basis \( \{ \beta^i \gamma^j \} \). This shape has several advantages:
\begin{itemize}
\item A left/right shift corresponds to multiplying each basis element by \( \beta \) or \( \beta^{-1} \), which updates only the new column — requiring the computation of just \( m \) new values.
\item An up/down shift corresponds to multiplying by \( \gamma \) or \( \gamma^{-1} \), which affects only the new row — requiring only \( \frac{n}{m} \) new values.
\item The rest of the pattern remains unchanged due to overlap with the previous state.
\end{itemize}

Thus, regular sampling patterns such as rectangles support partial updates, minimizing computation and memory usage. This makes them highly desirable in practical applications, such as image tiling, streaming data analysis, and memory-efficient coding schemes.

\section{Example}\label{examples}

Consider \( \mathbb{F}_{16} = \mathbb{F}_2(\alpha) \), where \( \alpha \) is a root of the irreducible polynomial \( x^4 + x + 1 \). Let \( \beta = \alpha^5 \) and \( \gamma = \alpha^3 \), so that \( \mathbb{F}_4 = \mathbb{F}_2(\beta) \), and note that \( \mathbb{F}_{16}^\times \) is generated by \( \alpha \). Since \( \gcd(5,3) = 1 \), the set \( \{ \alpha^{5i + 3j} \} \), as \( i,j \) range over \( \mathbb{Z}/3\mathbb{Z} \) and \( \mathbb{Z}/5\mathbb{Z} \), covers all nonzero elements of \( \mathbb{F}_{16} \) exactly once. We arrange these elements in a \( 3 \times 5 \) toroidal array indexed by row \( i \) and column \( j \).

The resulting array of field elements is:
\[
\begin{array}{|c|c|c|c|c|}
\hline
 \alpha^0 & \alpha^3 & \alpha^6 & \alpha^9 & \alpha^{12} \\
\hline
 \alpha^5 & \alpha^8 & \alpha^{11} & \alpha^{14} & \alpha^2 \\
\hline
 \alpha^{10} & \alpha^{13} & \alpha^1 & \alpha^4 & \alpha^7 \\
\hline
\end{array}
\]

Applying the trace function
\[
\mathrm{tr}_{\mathbb{F}_{16}/\mathbb{F}_2}(x) = x + x^2 + x^4 + x^8,
\]
entrywise to this array yields the corresponding binary torus:
\[
\begin{array}{|c|c|c|c|c|}
\hline
0 & 1 & 1 & 1 & 1 \\
\hline
0 & 0 & 1 & 1 & 0 \\
\hline
0 & 1 & 0 & 0 & 1 \\
\hline
\end{array}
\]

This binary array forms a trace-based De Bruijn torus over \( \mathbb{F}_2 \), constructed via the coordinate map \( x \mapsto \psi(x), \psi(\alpha x), \dots \), where \( \psi = \mathrm{tr}_{\mathbb{F}_{16}/\mathbb{F}_2} \). By the occurrence criterion in Theorem~\ref{thm:occurrence-criterion}, each column corresponds to a cyclic shift of the standard sequence \( DB_\beta \), and the total number of distinct columns matches the number of irreducible quadratics \( x^2 + \beta^i x + 1 \) over \( \mathbb{F}_4 \).

Indeed, the nonzero columns are cyclic shifts of the sequence
\[
DB_{\beta} = 
\begin{array}{|c|}
\hline
0 \\
\hline
1 \\
\hline
1 \\
\hline
\end{array},
\]
namely \( DB_\beta^{[1]} \) and \( DB_\beta^{[2]} \), reflecting the fact that the polynomials \( x^2 + \beta^i x + 1 \) are irreducible over \( \mathbb{F}_4 \) for \( i = 1, 2 \).

We now consider various sampling patterns—i.e., subsets of positions in the torus—used to evaluate the trace function. A pattern is valid if the corresponding evaluations form a linearly independent set over \( \mathbb{F}_2 \). Below are some typical examples:

$$
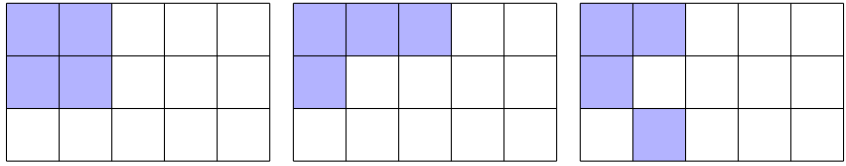
\begin{tikzpicture}[scale=0.7]
  % Fill a 2x2 block starting at (1,1)
  \foreach \i in {0,1} {
    \foreach \j in {1,2} {
      \fill[blue!30] (\i,\j) rectangle ++(1,1);
    }
  }

  % Draw 3x5 grid
  \foreach \i in {0,1,2,3} {
    \draw (0,\i) -- (5,\i);
  }
  \foreach \j in {0,1,2,3,4,5} {
    \draw (\j,0) -- (\j,3);
  }
\end{tikzpicture}
\hspace{.3cm}
\begin{tikzpicture}[scale=0.7]
  \foreach \i in {0,1,2} {
      \fill[blue!30] (\i,2) rectangle ++(1,1);
    }
  \fill[blue!30] (0,1) rectangle ++(1,1);

  % Draw 3x5 grid
  \foreach \i in {0,1,2,3} {
    \draw (0,\i) -- (5,\i);
  }
  \foreach \j in {0,1,2,3,4,5} {
    \draw (\j,0) -- (\j,3);
  }
\end{tikzpicture}
\hspace{.3cm}
\begin{tikzpicture}[scale=0.7]
  \fill[blue!30] (0,1) rectangle ++(1,1);
    \fill[blue!30] (0,2) rectangle ++(1,1);
    \fill[blue!30] (1,2) rectangle ++(1,1);
      \fill[blue!30] (1,0) rectangle ++(1,1);
  % Draw 3x5 grid
  \foreach \i in {0,1,2,3} {
    \draw (0,\i) -- (5,\i);
  }
  \foreach \j in {0,1,2,3,4,5} {
    \draw (\j,0) -- (\j,3);
  }
\end{tikzpicture}
$$

Moreover, once a single independent pattern is chosen—such as
% (Insert minimal independent pattern with TikZ)
$$
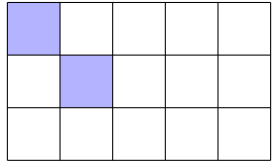
\begin{tikzpicture}[scale=0.7]
  \fill[blue!30] (1,1) rectangle ++(1,1);
  \fill[blue!30] (0,2) rectangle ++(1,1);
  % Draw 3x5 grid
  \foreach \i in {0,1,2,3} {
    \draw (0,\i) -- (5,\i);
  }
  \foreach \j in {0,1,2,3,4,5} {
    \draw (\j,0) -- (\j,3);
  }
\end{tikzpicture}$$

additional sampling patterns can be constructed by placing suitable disjoint translations of the same shape throughout the torus:
% (Insert TikZ: disjoint two-pattern copies)

$$
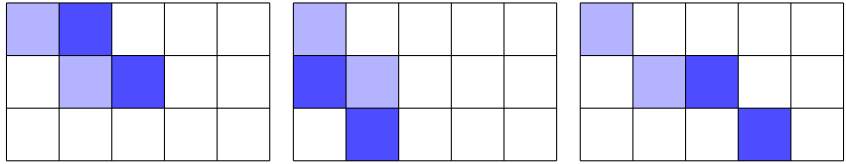
\begin{tikzpicture}[scale=0.7]
  \fill[blue!30] (1,1) rectangle ++(1,1);
  \fill[blue!30] (0,2) rectangle ++(1,1);
    \fill[blue!70] (2,1) rectangle ++(1,1);
  \fill[blue!70] (1,2) rectangle ++(1,1);
  % Draw 3x5 grid
  \foreach \i in {0,1,2,3} {
    \draw (0,\i) -- (5,\i);
  }
  \foreach \j in {0,1,2,3,4,5} {
    \draw (\j,0) -- (\j,3);
  }
\end{tikzpicture}
\hspace{.3cm}
\begin{tikzpicture}[scale=0.7]
  \fill[blue!30] (1,1) rectangle ++(1,1);
  \fill[blue!30] (0,2) rectangle ++(1,1);
    \fill[blue!70] (1,0) rectangle ++(1,1);
  \fill[blue!70] (0,1) rectangle ++(1,1);
  % Draw 3x5 grid
  \foreach \i in {0,1,2,3} {
    \draw (0,\i) -- (5,\i);
  }
  \foreach \j in {0,1,2,3,4,5} {
    \draw (\j,0) -- (\j,3);
  }
\end{tikzpicture}
\hspace{.3cm}
\begin{tikzpicture}[scale=0.7]
  \fill[blue!30] (1,1) rectangle ++(1,1);
  \fill[blue!30] (0,2) rectangle ++(1,1);
    \fill[blue!70] (3,0) rectangle ++(1,1);
  \fill[blue!70] (2,1) rectangle ++(1,1);
  % Draw 3x5 grid
  \foreach \i in {0,1,2,3} {
    \draw (0,\i) -- (5,\i);
  }
  \foreach \j in {0,1,2,3,4,5} {
    \draw (\j,0) -- (\j,3);
  }
\end{tikzpicture}
$$

This approach supports the recursive extension of patterns while preserving independence.

Finally, we illustrate how to compute a recursive update rule based on a \( 2 \times 2 \) sampling pattern. Suppose the current block values are:
\[
\begin{aligned}
B_{i,j} &= \psi(x), \quad &B_{i,j+1} &= \psi(x\alpha^3), \\
B_{i+1,j} &= \psi(x\alpha^5), \quad &B_{i+1,j+1} &= \psi(x\alpha^8).
\end{aligned}
\]
To compute the next row via a right shift, we evaluate:
\[
B_{i,j+2} = \psi(x\alpha^6), \quad B_{i+1,j+2} = \psi(x\alpha^{11}).
\]

$$
\begin{tikzpicture}[scale=1.2]
  % Fill a 2x2 block starting at (1,1)
  \foreach \i in {1,2} {
    \foreach \j in {1,2} {
      \fill[blue!30] (\i,\j) rectangle ++(1,1);
    }
  }
  \fill[red!40] (3,1) rectangle ++(1,1);
  \fill[red!40] (3,2) rectangle ++(1,1);

  % Draw 3x5 grid
  \foreach \i in {0,1,2,3} {
    \draw (0,\i) -- (5,\i);
  }
  \foreach \j in {0,1,2,3,4,5} {
    \draw (\j,0) -- (\j,3);
  }

  % Labels for B_{i,j}, etc.
  \node at (1.5, 2.5) {\small $B_{i,j}$};
  \node at (2.5, 2.5) {\small $B_{i,j+1}$};
  \node at (1.5, 1.5) {\small $B_{i+1,j}$};
  \node at (2.5, 1.5) {\small $B_{i+1,j+1}$};
  \node at (3.5, 2.5) {\small $B_{i,j+2}$};
  \node at (3.5, 1.5) {\small $B_{i+1,j+2}$};
\end{tikzpicture}
$$

Using the relations:
\[
\alpha^6 = 1 + \alpha^3 + \alpha^8, \quad \alpha^{11} = \alpha^3 + \alpha^5,
\]
we obtain the recursive rule:
\[
\begin{pmatrix}
B_{i,j+2} \\
B_{i+1,j+2}
\end{pmatrix}
=
\begin{pmatrix}
1 & 1 & 0 & 1 \\
0 & 1 & 1 & 0
\end{pmatrix}
\begin{pmatrix}
B_{i,j} \\
B_{i,j+1} \\
B_{i+1,j} \\
B_{i+1,j+1}
\end{pmatrix}.
\]

The matrix  
\[
m_{\mathrm{right}} = \begin{pmatrix}
1 & 1 & 0 & 1 \\
0 & 1 & 1 & 0
\end{pmatrix}
\]  
defines a linear rule for shifting a \( 2 \times 2 \) block one column to the right, reflecting the recursive structure induced by the trace map over \( \mathbb{F}_{16} \).

\section{Applications}

\subsection{Practical Applications}
As mentioned in the introduction, the concept of the \emph{De Bruijn torus} originates from real-world applications. However, the toroidal structure itself often presents challenges that hinder its direct implementation. In practice, these challenges can be resolved through simple adjustments based on the shape of the accompanying sampling pattern, thereby enabling feasible applications.

Let $B = (B_{i,j})$ be an $s \times t$ De Bruijn torus. For a sampling pattern $S$, define:
\begin{itemize}
    \item $I = \max\{i \mid (i, j) \in S\} - \min\{i \mid (i, j) \in S\}$
    \item $J = \max\{j \mid (i, j) \in S\} - \min\{j \mid (i, j) \in S\}$
\end{itemize}

By applying a toroidal translation, we may assume that:
\[ \min\{i \mid (i, j) \in S\} = \min\{j \mid (i, j) \in S\} = 0. \]
In this case, $I = \max\{i \mid (i, j) \in S\}$ and $J = \max\{j \mid (i, j) \in S\}$. A suitable adjustment is to extend the $s \times t$ De Bruijn torus $B$ to a $(s + I - 1) \times (t + J - 1)$ array $B' = (B'_{i,j})$, defined by:
\[ B'_{i,j} = B_{i \bmod s,\ j \bmod t}. \]

For example, the extended array of the binary torus with the sampling pattern corresponding to the Kronecker basis in Section~\ref{examples} is:
\[
\begin{array}{|c|c|c|c|c|c|}
\hline
0 & 1 & 1 & 1 & 1 & 0 \\
\hline
0 & 0 & 1 & 1 & 0 & 0 \\
\hline
0 & 1 & 0 & 0 & 1 & 0 \\
\hline
0 & 1 & 1 & 1 & 1 & 0 \\
\hline
\end{array}
\]

And the one corresponding to the second L-shape sampling pattern is:
\[
\begin{array}{|c|c|c|c|c|c|c|}
\hline
0 & 1 & 1 & 1 & 1 & 0 & 1 \\
\hline
0 & 0 & 1 & 1 & 0 & 0 & 0 \\
\hline
0 & 1 & 0 & 0 & 1 & 0 & 1 \\
\hline
0 & 1 & 1 & 1 & 1 & 0 & 1 \\
\hline
\end{array}
\]

This also demonstrates that regular sampling patterns minimize memory usage, as discussed in Section~\ref{eff-update}.

\subsection{Generalizations to Higher Dimensions}
In Section~\ref{construction}, we showed that for a prime number $p$, a finite field $\mathbb{F}_{p^n}$ satisfying $p^n - 1 = s \cdot t$ with $\gcd(s, t) = 1$ leads to an $s \times t$ De Bruijn torus over $\mathbb{F}_p$.

This finite-field-based construction can be generalized to a De Bruijn $N$-torus. Let $p$ be a prime, and suppose $\mathbb{F}_{p^n}$ satisfies $p^n - 1 = \prod^{N}_{i=1} P_{i}$, where $P_{i}$ are pairwise relatively prime.

Define $\hat{P_k} = \prod_{1 \le i \le N, i \ne k} P_i$. Let $\alpha$ be a generator of $\mathbb{F}_{p^n}^\times$, and define $\beta_j = \alpha^{\hat{P_j}}$, so that $\langle \beta_j \rangle$ is a subgroup of order $P_j$.

We construct a $P_1 \times P_2 \times \cdots \times P_N$ toroidal tensor $A = (A_{i_1, \ldots, i_N})$ with entries in $\mathbb{F}_{p^n}^\times$, defined by:
\[
A_{i_1, \ldots, i_N} = \prod_{j=1}^{N} \beta_j^{i_j}.
\]

Since the $P_i$ are pairwise relatively prime, the map $ (i_1, \ldots, i_N) \mapsto \prod_{j=1}^{N} \beta_j^{i_j} $ is a bijection between positions and elements of $\mathbb{F}_{p^n}^\times$.

After applying a nonzero $\mathbb{F}_p$-linear map $\psi$ entrywise to the tensor $A$, we obtain a new tensor $B = (\psi(A_{i_1, \ldots, i_N}))$ with entries in $\mathbb{F}_p$. This tensor $B$ is referred to as a De Bruijn $N$-torus.

This generalization preserves the bijective encoding and structural advantages of the 2D case, while enabling applications in higher-dimensional settings such as multidimensional coding, spatial sampling, and combinatorial design.

\clearpage

\bibliographystyle{unsrt} 
\bibliography{reference}

\end{document}